\documentclass[a4paper,12pt]{article}

\usepackage{setspace}
\usepackage[utf8]{inputenc}
\usepackage{bm}
\usepackage{amssymb}
\usepackage{amsmath}
\usepackage{amsfonts}
\usepackage{amstext}
\usepackage{multicol}
\usepackage{mathdots}
\usepackage{graphicx}
\usepackage{url}
\usepackage{hyperref}
\usepackage{mathrsfs}
\usepackage{lscape}
\usepackage{enumerate}
\usepackage{amsthm}
\usepackage{mathtools}

\newcommand{\Mod}[1]{\ (\mathrm{mod}\ #1)}
\newcommand{\mexover}[1]{ \overline{\text{mex}}({#1})}
\newcommand{\smexover}[1]{\sigma \overline{\text{mex}}({#1})}
\newcommand{\mexrover}[1]{ \overline{\text{mex}}_r({#1})}
\newcommand{\srmexover}[1]{\sigma_r \overline{\text{mex}}({#1})}
\newtheorem{theorem}{Theorem}[section]

\newtheorem{definition}{Definition}

\newtheorem{proposition}{Proposition}

\title{Minimal Excludant over Overpartitions}
\author{Victor Manuel R. Aricheta\thanks{Institute of Mathematics, University of the Philippines, Diliman, Quezon
City 1101, Philippines; \texttt{vmaricheta@math.upd.edu.ph}} \,\,
and Judy Ann L. Donato\thanks{Institute of Mathematics, University of the Philippines, Diliman, Quezon
City 1101, Philippines; \texttt{jadonato@math.upd.edu.ph}}}
\date{}

\begin{document}
\onehalfspacing
\maketitle

\begin{abstract}
Define the minimal excludant of an overpartition $\pi$, denoted $\mexover{\pi}$, to be the smallest positive integer that is not a part of the non-overlined parts of $\pi$. For a positive integer $n$, the function $\smexover{n}$ is the sum of the
minimal excludants over all overpartitions of $n$. In this paper, we proved that the $\smexover{n}$ equals the number of partitions of $n$ into distinct parts using three colors. We also provide an asymptotic formula for $\smexover{n}$ and show that $\smexover{n}$ is almost always even and is odd exactly when $n$ is a triangular number. Moreover, we generalize $\mexover{\pi}$ using the least $r$-gaps, denoted $\mexrover{\pi}$, defined as the smallest part of the non-overlined parts of the overpartition $\pi$ appearing less than $r$ times. Similarly, for a positive integer $n$, the function $\srmexover{n}$ is the sum of the least $r$-gaps over all overpartitions of $n$. We derive a generating function and  an asymptotic formula for $\srmexover{n}$. Lastly, we study the arithmetic density of $\srmexover{n}$ modulo $2^k$, where $r=2^m\cdot3^n, m,n \in \mathbb{Z}_\geq 0.$\\
\ \ \ \ \

\noindent \textbf{Keywords}: partitions, minimal excludant, least gaps, modular forms

\noindent \textbf{AMS\ Classification: }05A17, 11F11, 11F20, 11P83
\end{abstract}


\section{Introduction}
The minimal excludant (mex) of a subset $S$ of a well-ordered set $U$ is the smallest value in $U$ that is not in $S$.  In particular, the minimal excludant of a set $S$ of positive integers, denoted mex($S$), is the least positive integer not in $S$, i.e., $\text{mex}(S)=\text{min}(\mathbb{Z}^+ \backslash S).$ The history of the minimal excludant goes way back in the 1930's when it was first used in combinatorial game theory by Sprague and Grundy \cite{sprague}, \cite{grundy}.\\

In 2019, Andrews and Newman \cite{andrewsnewman} studied the minimal excludant of an integer partition $\pi$, denoted $\text{mex}(\pi)$, which is defined as the smallest positive integer that is not a part of $\pi.$ Moreover, the arithmetic function
\[\sigma \text{mex}(n):=\displaystyle \sum_{\pi \in \mathcal{P}(n)} \text{mex}(\pi),\]
where $\mathcal{P}(n)$ is the set of all partitions of $n$, was also introduced.\\

They proved the following interesting relationship between $\sigma \text{mex}(n)$ and $D_2(n)$ which is the number of partitions of $n$ into distinct parts using two colors:
\[\sigma \text{mex}(n)=D_2(n).\]

Moreover, it was shown that $\sigma \text{mex}(n)$ is almost always even and is odd exactly when $n=j(3j\pm 1)$ for some $j \in \mathbb{N}.$\\

In \cite{ballantinemerca}, Ballantine and Merca explored the least $r$-gap of a partition $\pi$, denoted $g_r(\pi)$, which is the smallest part of $\pi$ appearing less than $r$ times. In particular, $g_1(\pi)$ is the minimal excludant of $\pi.$ They defined the arithmetic function 
\[\sigma_r \text{mex}(n)=\displaystyle \sum_{\pi \in \mathcal{P}(n)} g_r(\pi)\]
which is the sum of the least $r$-gaps in all partitions of $n$. The following generating function for $\sigma_r\text{mex}(n)$  was also derived:
\[\displaystyle \sum_{n=0}^\infty \sigma_r\text{mex}(n)q^n=\dfrac{(q^{2r};q^{2r})_\infty}{(q;q)_\infty(q^r;q^{2r})_\infty}.\]

Furthermore, Chakraborty and Ray \cite{chakrabortyray} studied the arithmetic density of $\sigma_2\text{mex}(n)$ and $\sigma_3\text{mex}(n)$ modulo $2^k$ for any positive integer $k$ and proved that for almost every nonnegative integer $n$ lying in an arithmetic progression, the integer
$\sigma_r \text{mex}(n) $is a multiple of $2^k$ where $r \in \{2,3\}.$\\

Now, recall that the overpartition of a positive integer $n$ is a non-increasing sequence of natural numbers whose sum is $n$ in which the first occurrence  of a number may be overlined. We denote by $\overline{p}(n)$ the number of overpartitions of $n$. For example, $\overline{p}(3)=8$ since there are 8 overpartitions of $3$ which are: 
\[3, \overline{3}, 2+1, \overline{2}+1, 2+\overline{1}, \overline{2}+\overline{1}, 1+1+1, \overline{1}+1+1.\]

The goal of this paper is to extend the notion of minimal excludant of partitions to overpartitions. There are several ways to obtain such generalization, but we propose the  following definition below (Definition \ref{def2}). We justify using this definition through the results we obtain, results that are manifestly analogues of results concerning the classical partition function (see Proposition \ref{Tk} for example).

\begin{definition}
    The \textbf{minimal exludant of an overpartition} $\pi$, denoted $\overline{\text{mex}}(\pi)$, is the smallest positive integer that is not a part of the non-overlined parts of $\pi$.  For a positive integer $n$, denote the sum of $\overline{\text{mex}}(\pi)$ over all overpartitions $\pi$ of $n$ as $\sigma \overline{\text{mex}}(n):$
\[\sigma \overline{\text{mex}}(n)=\displaystyle \sum_{\pi \in \overline{\mathcal{P}}(n)} \overline{\text{mex}}(\pi), \]
where $\overline{\mathcal{P}}(n)$ is the set of all overpartitions of $n$. Moreover, we set $\smexover{0}=1.$
\end{definition}


\noindent For example, consider $n=3.$ The table below shows all overpartitions of $3$ and their corresponding minimal excludant.\\
\begin{center}
\begin{tabular}{ |c|c| } 
 \hline
 $\pi$ & $\overline{\text{mex}}(\pi)$\\
 \hline
 3 & 1 \\ 
$\overline{3}$ & 1 \\ 
  $2 + 1$ & 3 \\
  $\overline{2} + 1$ & 2 \\
   $2 + \overline{1}$ & 1 \\
    $\overline{2} + \overline{1}$ & 1\\
   1+1+1 & 2\\
   $\overline{1}+1+1$ & 2 \\
 \hline 
\end{tabular}
\end{center}

\noindent Thus, $\sigma \overline{\text{mex}}(3)= 13.$
\\

In Section 2, we derive the generating function of $\smexover{n}$ and prove the following theorem relating $\smexover{n}$ and $D_3(n)$, the number of partitions of $n$ into distinct parts using three colors. This theorem may be viewed as a generalization of the results of Andrews and Newman, which relates $\sigma\text{mex}(n)$ and $D_2(n).$

\begin{theorem} \label{D3}
For all positive integer $n$, we have \[\sigma\overline{\text{mex}}(n)=D_3(n).\]
\end{theorem}

We then derive an asymptotic formula for $\smexover{n}$ and prove a theorem regarding the parity of $\smexover{n}$.

\begin{theorem} \label{asym1}
We have
\[\smexover{n} \sim \dfrac{e^{\pi\sqrt{n}}}{8n^{3/4}}\] as $n \rightarrow \infty.$
\end{theorem}

\begin{theorem} \label{parity}
For a positive integer $n$, we have 
\[\smexover{n} \equiv \begin{cases}
    1 \mod 2, \text{ if } n=\frac{j(j+1)}{2} \text{ for some } j \in \mathbb{N}\\
    0 \mod 2, \text{ otherwise.}
\end{cases}\]
\end{theorem}

In Section 3, we generalize $\mexover{n}$ into $r$-gaps which we define as follows.

\begin{definition} \label{def2}
The \textbf{least r-gap of an overpartition $\pi$}, denoted $\mexrover{\pi}$ is the smallest part of the non-overlined parts of $\pi$ appearing less than $r$ times. Moreover, the function
 \[\srmexover{n} = \displaystyle \sum_{\pi \in \overline{\mathcal{P}}(n)} \mexrover{\pi} \]
is the  sum of the least $r$-gaps over all overpartitions of $n.$ 
\end{definition}

We then derive the generating function and an asymptotic formula for $\srmexover{n}$.

          \begin{theorem} \label{gen}
        For all positive integer $n$, we have
        \[\displaystyle\sum_{n=0}^{\infty} \srmexover{n}q^n = \dfrac{(-q;q)_\infty(q^{2r};q^{2r})_\infty}{(q;q)_\infty(q^{r};q^{2r})_\infty}.\]\end{theorem}

\begin{theorem} \label{asym2}
We have
    \[\srmexover{n} \sim \dfrac{e^{\pi\sqrt{n}}}{8rn^{3/4}}\] \text{ as } $ n \rightarrow \infty.$ 
\end{theorem}

Lastly, we prove our main result about the distribution of $\srmexover{n}$ when $r=2^m \cdot 3^n, m,n \in \mathbb{Z}_{\geq 0}$ in Section 4.

\begin{theorem} \label{main}
    Let $r = 2^m \cdot 3^n$ where $m,n \in \mathbb{Z}_{\geq 0}$ and  $k \geq 1$ be a positive integer. Then 
\[\displaystyle \lim_{X \to +\infty} \dfrac{\#\{n\leq X: \sigma_r\overline{mex}(n) \equiv 0 \Mod {2^k}\}}{X}=1.\]
\end{theorem}

Equivalently, for almost every nonnegative integer $n$ lying in an arithmetic progression, the integer
$\srmexover{n}$ is a multiple of $2^k$ when $r=2^m \cdot 3^n, m,n \in \mathbb{Z}_{\geq 0}$.


\section{Minimal Excludant of an Overpartition}
\subsection{Generating Function of $\sigma\overline{\text{mex}}(n)$}

\noindent \large \textbf{Proof of Theorem \ref{D3}:}

\normalsize
\begin{proof}
Let $\overline{p^{\text{mex}}}(m,n)$ be the number of overpartitions $\pi$ of $n$ with $\overline{\text{mex}}(\pi)=m.$
Then we have the following double series $M(z,q)$ in which the coefficient of $z^mq^n$ is  $\overline{p^{\text{mex}}}(m,n)$:
\begin{align*}
M(z,q):=\displaystyle\sum_{n=0}^\infty \sum_{m=1}^\infty \overline{p^{\text{mex}}}(m,n) z^mq^n &= \displaystyle \sum_{m=1}^\infty z^mq^1\cdot q^2 \cdots \cdot q^{m-1} \cdot \dfrac{\displaystyle\prod_{n=1}^\infty (1+q^n)}{\displaystyle\prod_{\substack{n-1 \\ n \neq m}}^\infty (1-q^n)}\\
    &=\displaystyle \sum_{m=1}^\infty z^mq^{m \choose 2} \cdot \dfrac{(-q;q)_\infty}{(q;q)_\infty} \cdot (1-q^m)\\
    &=\dfrac{(-q;q)_\infty}{(q;q)_\infty} \displaystyle \sum_{m=1}^\infty z^mq^{m \choose 2} \cdot (1-q^m)
\end{align*}

Thus,
\begingroup
\allowdisplaybreaks
\begin{align*}
    \displaystyle \sum_{n \geq 0} \sigma \overline{\text{mex}}(n)q^n &= \dfrac{\partial}{\partial z}\Big|_{z=1} M(z,q)\\
    &=\dfrac{(-q;q)_\infty}{(q;q)_\infty} \displaystyle \sum_{m=0}^\infty mq^{m \choose 2} (1-q^m)\\
     &=\dfrac{(-q;q)_\infty}{(q;q)_\infty} \left( \displaystyle \sum_{m=1}^\infty mq^{m \choose 2} -\displaystyle \sum_{m=1}^\infty mq^{m \choose 2}\cdot q^m\right)\\
    &= \dfrac{(-q;q)_\infty}{(q;q)_\infty} \left( \displaystyle \sum_{m=1}^\infty mq^{m \choose 2} -\displaystyle \sum_{m=1}^\infty (m-1)q^{m \choose 2}\right)\\
    &=\dfrac{(-q;q)_\infty}{(q;q)_\infty} \displaystyle \sum_{m=0}^\infty q^{m+1 \choose 2}\\
    &=\dfrac{(-q;q)_\infty}{(q;q)_\infty} \cdot \dfrac{(q^2;q^2)_\infty}{(q;q^2)_\infty}\\
    &=(-q;q)_\infty \cdot \dfrac{(q^2;q^2)_\infty}{(q;q)_\infty(q;q^2)_\infty}\\
    &=(-q;q)_\infty \cdot (-q;q)^2_\infty\\
    &=(-q;q)^3_\infty\\
    &=\displaystyle \sum_{n \geq 0} D_3(n)q^n.
\end{align*}
\endgroup
\end{proof}

As an illustration, observe that the thirteen 3-colored partitions of 3 are: $3_1, 3_2, 3_3, 2_1+1_1, 2_1+1_2, 2_1+1_3, 2_2+1_1, 2_2+1_2, 2_2+1_3, 2_3+1_1, 2_3+1_2, 2_3+1_3, 1_1+1_2+1_3.$ Indeed, $D_3(3)=13=\sigma\overline{\text{mex}}(3).$


\subsection{Asymptotic Formula for $\smexover{n}$}

To derive an asymptotic formula for $\smexover{n}$, we will be using the following asymptotic result by Ingham \cite{ingham} about the coefficients of a power series.

\begin{proposition} \label{asymingham}
  Let $A(q)= \sum_{n=0}^{\infty} a(n)q^n$ be a power series with radius of convergence 1. Assume that $\{a(n)\}$ is a weakly increasing sequence of nonnegative real numbers. If there are constants $\alpha, \beta \in \mathbb{R},$ and $C>0$ such that
\[A(e^{-t})\sim \alpha t^{\beta}e^{\frac{C}{t}}, \text{ as } t \rightarrow 0^{+}\]
then we have 
\[a(n) \sim \dfrac{\alpha}{2\sqrt{\pi}}\dfrac{C^{\frac{2\beta+1}{4}}}{n^{\frac{2\beta+3}{4}}}e^{2\sqrt{Cn}}, \text{ as } n \rightarrow \infty.\]
\end{proposition}

\noindent \large \textbf{Proof of Theorem \ref{asym1}:}
\normalsize
\begin{proof}
Note that $\smexover{n}=D_3(n)$ and $\{D_3(n)\}$ is an increasing sequence of nonnegative real numbers, thus $\smexover{n}$ is also an increasing sequence of nonnegative real numbers.\\

\noindent Let $A(q)=(-q;q)^3_\infty$, where $a(n)=\smexover{n}$ as in Proposition \ref{asymingham}.\\

\noindent From \cite{bhorjaetal},
\begin{align}\dfrac{1}{(e^{-t}; e^{-t})_\infty} \sim \sqrt{\dfrac{t}{2\pi}} e^{\frac{\pi^2}{6t}} \text{ as } t \rightarrow 0^+. \label{1}
\end{align}
Moreover, we will use the following identity
\begin{align}
(-q;q)_{\infty}=\dfrac{1}{(q;q^2)_\infty}=\dfrac{(q^2;q^2)_\infty}{(q;q)_\infty}.\label{2}
\end{align}

\noindent By (\ref{1}) and (\ref{2}), as $t \rightarrow 0^{+},$
\[(-e^{-t};e^{-t})_\infty = \dfrac{(e^{-2t};e^{-2t})_\infty}{(e^{-t};e^{-t})_\infty}\sim \frac{\sqrt{\frac{t}{2\pi}} e^{\frac{\pi^2}{6t}}}{\sqrt{\frac{2t}{2\pi}} e^{\frac{\pi^2}{12t}}}=\dfrac{1}{\sqrt{2}}e^{\frac{\pi^2}{12t}}.\]

\noindent Hence, as $t \rightarrow 0^{+},$
\begin{align}
A(e^{-t})=(-e^{-t}; e^{-t})_\infty^3 \sim \left(\dfrac{1}{\sqrt{2}}e^{\frac{\pi^2}{12t}}\right)^3 = \dfrac{1}{2\sqrt{2}}e^{\frac{\pi^2}{4t}}. \label{3}
\end{align}

\noindent Take $\alpha=\dfrac{1}{2\sqrt{2}}, \beta=0$ and $C=\frac{\pi^2}{4}$, by Proposition \ref{asymingham},
\[\smexover{n} \sim \dfrac{\frac{1}{2\sqrt{2}}}{2\sqrt{\pi}} \dfrac{\left(\frac{\pi^2}{4}\right)^{1/4}}{n^{3/4}}e^{2\sqrt{\frac{\pi^2}{4}n}} = \dfrac{e^{\pi{\sqrt{n}}}}{8n^{3/4}} \]
as $n \rightarrow \infty$.
\end{proof}

\newpage
\subsection{Parity of $\smexover{n}$}
\noindent \large \textbf{Proof of Theorem \ref{parity}:}
\normalsize
\begin{proof}
We have
    \begin{align*}
        \displaystyle \sum_{n \geq 0} \smexover{n}q^n   &=(-q;q)_\infty^3 \\
        &= \displaystyle \prod_{n=1}^{\infty} (1+q^n)^3\\
        &\equiv \displaystyle \prod_{n=1}^{\infty} (1-q^n)^3 \Mod 2\\
        &=(q;q)_\infty^3\\
        &=\displaystyle \sum_{j=0}^\infty (-1)^j(2j+1)q^{\frac{j(j+1)}{2}}, \text{ by Jacobi's identity \cite{jacobi}.} 
    \end{align*}
\end{proof}

Comparing coefficients, we have that $\smexover{n} \equiv 0 \Mod 2$ for $n \neq \frac{j(j+1)}{2}$ for some $j \in \mathbb{N}$ and $\smexover{n} \equiv 1 \Mod 2$ otherwise. This shows that $\smexover{n}$ is almost always even and is odd exactly when $n$ is a triangular number.
\section{Least $r$-gaps}

\subsection{Generating Function of $\srmexover{n}$}

In \cite{ballantinemerca}, Ballantine and Merca proved that for $n \geq 0$ and $r \geq 1$,
\[\displaystyle \sum_{k=0}^{\infty} p(n-rT_k) = \sigma_r \text{mex}(n).\]

We extend this result to overpartitions and present an analogous proof for the following proposition.
    \begin{proposition} \label{Tk}
        For $n \geq 0$ and $r \geq 1,$
        \[\displaystyle\sum_{k=0}^{\infty} \overline{p}(n-rT_k)=\srmexover{n}.\]  
    \end{proposition}

    \begin{proof}
        Fix $r \geq 1,$ for each $k \geq 0,$  consider the staircase partition 
        \[\delta_r(k)=(1^r, 2^r,...,  (k-1)^r, k^r)\]
        where each part from 1 to $k$ is repeated $r$ times. We create an injection from the set of overpartitions of $n-rT_k$ into the set of of overpartitions of $n$ with the following mapping:
        \[\phi_{r,n,k}: \overline{\mathcal{P}}(n-rT_k)  \xhookrightarrow{} \overline{\mathcal{P}}(n)\]
        where for an overpartition $\pi$ of $n-rT_k$, $\phi_{n,r,k}(\pi)$ is the overpartition obtained by inserting the non-overlined staircase partition $\delta_r(k)$.\\
    
        For example, if $\pi= 4+\overline{3}+2+\overline{1}+1=11$, we have $\phi_{2,11,3}=4+\overline{3}+2+\overline{1}+1+3+3+2+2+1+1=23.$\\
        
        Let $\mathcal{A}_{r,n,k}$ be the image of the overpartitions of $n-rT_k$ under $\phi_{n,r,k}$. We have $\overline{p}(n-rT_k)=|\mathcal{A}_{r,n,k}|$ and $\mathcal{A}_{r,n,k}$ consists of the partitions of $n$ satisfying $\mexrover{\pi}>k$.\\

        Now, suppose $\pi$ is an overpartition of $n$ with $\mexrover{\pi}=k$. Then $\pi \in \mathcal{A}_{r,n,i},$ for $ i=0,1,...,k-1$ and $\pi \notin A_{r,n,j}$ with $j \geq k.$ Thus, each overpartition of $n$ with $\mexrover{\pi}=k$ is counted by the summation $\displaystyle\sum_{k=0}^{\infty} \overline{p}(n-rT_k)$ exactly $k$ times.
        
    \end{proof}

\newpage
    \noindent \large \textbf{Proof of Theorem \ref{gen}:}
    \normalsize
        
    \begin{proof} We have
\begin{align*}
    \displaystyle \sum_{n=0}^{\infty} \srmexover{n}q^n &= \displaystyle \sum_{n=0}^{\infty}
\left(\displaystyle\sum_{k=0}^{\infty} \overline{p}(n-rT_k)\right)q^n, \text{ by Proposition \ref{Tk}}\\
  &= \displaystyle \sum_{n=0}^{\infty} 
\displaystyle\sum_{k=0}^{\infty} \overline{p}(n-rT_k)q^n\\
&=\displaystyle \sum_{n=0}^\infty \displaystyle \sum_{k=0}^\infty \overline{p}(n)q^{n+rT_k} \\
&=\left(\displaystyle \sum_{n=0}^\infty \overline{p}(n)q^n \right)\left(\displaystyle \sum_{k=0}^\infty q^{rT_k} \right).\\
\end{align*}

\noindent Note that the generating function for $\overline{p}(n)$ is
\[\displaystyle \sum_{n=0}^\infty \overline{p}(n)q^n = \dfrac{(-q;q)_\infty}{(q;q)_\infty}.\]
Moreover, from \cite{ballantinemerca},  
\[\displaystyle \sum_{k=0}^\infty q^{rT_k} = \dfrac{(q^{2r};q^{2r})_\infty}{(q^r;q^{2r})_\infty}.\]
Thus,
\[  \displaystyle \sum_{n=0}^{\infty} \srmexover{n}q^n = \left(\displaystyle \sum_{n=0}^\infty \overline{p}(n)q^n \right)\left(\displaystyle \sum_{k=0}^\infty q^{rT_k} \right) = \dfrac{(-q;q)_\infty(q^{2r};q^{2r})_\infty}{(q;q)_\infty(q^{r};q^{2r})_\infty}.\]
    \end{proof}


\subsection{Asymptotic Formula for $\srmexover{n}$}

Here, we generalize our asymptotic result in Theorem \ref{asym1} for least $r$-gaps.\\
\newpage
\noindent \large \textbf{Proof of Theorem \ref{asym2}:}
\normalsize
\begin{proof}
Note that $\overline{p}(n) < \overline{p}(n+1)$ for $n \in \mathbb{N},$ since for every overpartition of $n$, say $n=a_1+a_2+\cdots+a_l$, we correspondingly have $n+1=a_1+a_2+\cdots+a_l+1$ as an overpartition of $n+1.$ Since $\srmexover{n}$ is the sum of the least $r$-gaps taken over all overpartitions of $n$, then we can conclude that $\srmexover{n}$ is a weakly increasing sequence. \\

\noindent Let $A(q)=\dfrac{(-q;q)_\infty(q^{2r};q^{2r})_\infty}{(q;q)_\infty(q^{r};q^{2r})_\infty}$ , where $a(n)=\srmexover{n}$ as in Proposition \ref{asymingham}.\\

\noindent First,
\[A(q)=\dfrac{(-q;q)_\infty(q^{2r};q^{2r})_\infty}{(q;q)_\infty(q^{r};q^{2r})_\infty}=\dfrac{(-q;q)_\infty(-q^r;q^r)^2_\infty (q^r;q^r)_\infty}{(q;q)_\infty}.\]

\noindent Hence, using (\ref{3}), as $t \rightarrow 0^{+},$
\begin{align*}
    A(e^{-t}) & = \dfrac{(-e^{-t};e^{-t})_\infty(-e^{-rt};e^{-rt})_\infty^2(e^{-rt};e^{-rt})_\infty}{(e^{-t};e^{-t})_\infty}\\
    & \sim \dfrac{\frac{1}{\sqrt{2}}e^{\frac{\pi^2}{12t}}\left(\frac{1}{\sqrt{2}}e^{\frac{\pi^2}{12rt}}\right)^2\sqrt{\frac{t}{2\pi}}e^{\frac{\pi^2}{6t}}}{\sqrt{\frac{rt}{2\pi}}e^{\frac{\pi^2}{6rt}}}\\
    &=\dfrac{1}{2\sqrt{2}r}e^{\frac{\pi^2}{4t}}
\end{align*}

\noindent Take $\alpha=\dfrac{1}{2\sqrt{2}r}, \beta=0$ and $C=\frac{\pi^2}{4}$, by Proposition \ref{asymingham},
\[\srmexover{n} \sim \dfrac{\frac{1}{2\sqrt{2}r}}{2\sqrt{\pi}} \dfrac{\left(\frac{\pi^2}{4}\right)^{1/4}}{n^{3/4}}e^{2\sqrt{\frac{\pi^2}{4}n}} = \dfrac{e^{\pi{\sqrt{n}}}}{8rn^{3/4}} \]
as $n \rightarrow \infty$.
\end{proof}


\section{Distribution of $\srmexover{n}$}

\subsection{Preliminaries}
We first discuss some preliminaries about modular forms. We  define the upper-half complex plane 
\[\mathbb{H}=\{z \in \mathbb{C} \mid \text{Im}(z)>0\}\]
and the modular group 
    \[\text{SL}_2(\mathbb{Z}) = \left\{\begin{pmatrix}
      a & b\\
      c & d
  \end{pmatrix} \Big| \, ad-bc = 1; a,b,c,d \in \mathbb{Z}\right\}.\]

For $A={\begin{pmatrix}
      a & b\\
      c & d
  \end{pmatrix}} \in \text{SL}_2(\mathbb{Z}),$ the modular group $\text{SL}_2(\mathbb{Z})$ acts on $\mathbb{H}$ by the following linear fractional transformation: 
  \[Az= {\begin{pmatrix}
      a & b\\
      c & d
  \end{pmatrix}}z =\dfrac{az+b}{cz+d}.\]

\noindent Moreover, if $N \in \mathbb{Z}^+$, we define the following \textbf{congruence subgroups} of $SL_2(\mathbb{Z})$ of level $N$: 
\[\Gamma_0(N):=\left\{\begin{pmatrix}
      a & b\\
      c & d
  \end{pmatrix} \in \text{SL}_2(\mathbb{Z}) \Big| \begin{pmatrix}
      a & b\\
      c & d
  \end{pmatrix} \equiv \begin{pmatrix}
      * & *\\
      0 & *
  \end{pmatrix} \mod N \right\}\]
 \[ \Gamma_1(N):=\left\{\begin{pmatrix}
      a & b\\
      c & d
  \end{pmatrix} \in \text{SL}_2(\mathbb{Z}) \Big| \begin{pmatrix}
      a & b\\
      c & d
  \end{pmatrix} \equiv \begin{pmatrix}
      1 & *\\
      0 & 1
  \end{pmatrix} \mod N \right\}\]
  \[\Gamma(N):=\left\{\begin{pmatrix}
      a & b\\
      c & d
  \end{pmatrix} \in \text{SL}_2(\mathbb{Z}) \Big| \begin{pmatrix}
      a & b\\
      c & d
  \end{pmatrix} \equiv \begin{pmatrix}
      1 & 0\\
      0 & 1
  \end{pmatrix} \mod N\right\}.\]

\noindent Note that the following inclusions are true:
\[\Gamma(N) \subseteq \Gamma_1(N) \subseteq \Gamma_0(N) \subseteq \text{SL}_2(\mathbb{Z}).\\\]

 Modular forms are complex functions on $\mathbb{H}$ that transforms nicely under these congruence subgroups of $\text{SL}_2({\mathbb{Z}}).$ For this paper, we are interested on modular forms transforming nicely with respect to $\Gamma_0(N)$ having a  Nebentypus character $\chi$ defined as follows.

\begin{definition}
    Let $\chi$ be a Dirichlet character modulo $N$ (a positive integer). Then a modular form $f \in M_k(\Gamma_1(N))$ has Nebentypus character $\chi$ if 
    \[f\left(\dfrac{az+b}{cz+d}\right)=\chi(d)(cz+d)^k f(z)\]
    for all $z \in \mathbb{H}$ and all $\begin{pmatrix}
        a & b\\
        c & d
    \end{pmatrix} \in \Gamma_0(N)$. 
    The space of all such modular forms is denoted $M_k(N,\chi)$.
\end{definition}

In particular, we look at modular forms involving the Dedekind eta function which is defined as follows.
\begin{definition}The Dedekind eta function is the function $\eta(z)$ where $z \in  \mathbb{H}:$ 
    \[\eta(z) = e^{\frac{\pi iz}{12}} \prod_{n=1}^{\infty} (1-e^{2\pi i nz}).\]

   \noindent Defining $q:=e^{2\pi i z}$, we have:
    \[\eta(z)=q^{\frac{1}{24}} \prod_{n=1}^{\infty}(1-q^n).\]
\end{definition}

\begin{definition}
      A function $f(z)$ is called an \textbf{eta-product} if it is expressible as  a finite product of the form
    \[f(z)=\prod_{\delta \mid N} \eta^{r_\delta}(\delta z)\]
    where $N$ and each $r_\delta$ is an integer.
\end{definition}

The next two theorems will be used to proved that an eta-product is a holomorphic modular form.

\begin{theorem}[Gordon, Hughes, Newman] \label{gordon}
    If $f(z)= \prod_{\delta \mid N} \eta^{r_\delta}(\delta z)$ is an eta-product for which 
  \begin{align}\displaystyle \sum_{\delta \mid N} \delta r_{\delta} \equiv 0 \Mod {24} \label{4} \end{align}
and
\begin{align}
\displaystyle \sum_{\delta \mid N} \dfrac{N}{\delta}r_{\delta} \equiv 0 \Mod {24} \label{5}
\end{align}
then $f(z)$ satisfies
\[f(Az)=\chi(d)(cz+d)^kf(z)\]
for all $A=\begin{pmatrix}
a & b\\
c & d 
\end{pmatrix} \in \Gamma_0(N)$ where $k=\displaystyle \sum_{\delta \mid N} r_\delta.$ Here the character $\chi$ is defined by $\chi(d)=\left(\dfrac{(-1)^ks}{d}\right)$ and $s=\displaystyle\prod_{\delta \mid N} \delta^{r_\delta}.$
\end{theorem}

\begin{theorem}[Ligozat] \label{ligozat}
    Let $c,d$ and $N$ be positive integers with $d \mid N$ and $\text{gcd}(c,d)=1.$ With the notation as above, if the eta-product $f(z)$ satisfies (\ref{4}) and (\ref{5}), then the order of vanishing of $f(z)$ at the cusp $\frac{c}{d}$ is 
         \[\dfrac{1}{24}\displaystyle\sum_{\delta \mid N} \dfrac{N\text{gcd}(d,\delta)^2r_\delta}{\text{gcd}\left(d,\frac{N}{d}\right)d\delta}.\]
\end{theorem}

\subsection{Proof of Main Result}

Before we prove Theorem \ref{main}, we prove two propositions first. 
\begin{proposition}\label{fourier}
     Let $k$ be a positive integer. Then
    \[f_{r,k}(z):=\dfrac{\eta(48z)\eta(24rz)^{2^k-1}}{\eta(24z)^2\eta(48rz)^{2^{k-1}-2}} \equiv \displaystyle\sum_{n=0}^\infty \sigma_r\overline{\text{mex}}(n)q^{24n+3r} (\text{mod } 2^k) \]
\end{proposition}
   
\noindent \textbf{Proof:}
Consider $$g(z)=\dfrac{\eta(24rz)^2}{\eta(48rz)}= \dfrac{(q^{24r};q^{24r})^2_\infty}{(q^{48r};q^{48r})_\infty}.$$ \\
By the binomial theorem, $(q^r;q^r)^{2^k}_\infty \equiv (q^{2r};q^{2r})^{2^{k-1}}_\infty \Mod{2^k}.$ \\

\noindent Thus, $(q^{24r};q^{24r})^{2^k}_\infty \equiv (q^{48r};q^{48r})^{2^{k-1}}_\infty \Mod {2^k}$, and so \[g^{2^{k-1}}(z)=\dfrac{(q^{24r};q^{24r})^{2^k}_\infty}{(q^{48r};q^{48r})^{2^{k-1}}_\infty} \equiv 1 \Mod {2^k}.\]

\noindent Now, consider 
\begin{align*}
\dfrac{\eta(48z)\eta(48rz)^2}{\eta(24z)^2\eta(24rz)}\cdot  g^{2^{k-1}}(z)&= \dfrac{\eta(48z)\eta(48rz)^2}{\eta(24z)^2\eta(24rz)}\cdot \dfrac{\eta(24rz)^{2^k}}{\eta(48rz)^{2^k-1}}\\
&=\dfrac{\eta(48z)\eta(24rz)^{2^k-1}}{\eta(24z)^2\eta(48rz)^{2^{k-1}-2}}\\
&=f_{r,k}(z).
\end{align*}

\noindent Observe that
\begin{align*}
    f_{r,k}(z)&=\dfrac{\eta(48z)\eta(48rz)^2}{\eta(24z)^2\eta(24rz)}\cdot  g^{2^k-1}(z)\\
    &\equiv\dfrac{\eta(48z)\eta(48rz)^2}{\eta(24z)^2\eta(24rz)}  
  \Mod {2^k}\\
    &=q^{3r} \dfrac{(q^{48};q^{48})_\infty(q^{48r};q^{48r})^2_\infty}{(q^{24};q^{24})^2_\infty (q^{24r};q^{24r})_\infty}.
\end{align*}

\noindent Note that
\begin{align*}
    \displaystyle\sum_{n=0}^\infty \sigma_r\overline{\text{mex}}(n)q^n &= \dfrac{(-q;q)_\infty (q^{2r};q^{2r})_\infty}{(q;q)_\infty (q^{r};q^{2r})_\infty}\\
    &=\dfrac{(-q;q)_\infty (q^{2r};q^{2r})^2_\infty}{(q;q)_\infty (q^{r};q^{r})_\infty}\\
    &=\dfrac{(q^2;q^2)_\infty (q^{2r};q^{2r})^2_\infty}{(q;q)^2_\infty (q^{r};q^{r})_\infty}.
\end{align*}

Hence, \[f_{r,k}(z) \equiv q^{3r}  \displaystyle\sum_{n=0}^\infty \sigma_r\overline{\text{mex}}(n)q^{24n} \Mod {2^k} =  \displaystyle\sum_{n=0}^\infty \sigma_r\overline{\text{mex}}(n)q^{24n+3r}.\]

\begin{proposition}
    Let  $r=2^m \cdot 3^n$ where $m,n \in \mathbb{Z}_{\geq 0}$ and $k \geq m+2n+1$ be an integer greater than 3. Then
$f_{r,k}(z) \in M_{2^{k-2}}(\Gamma_0(N),\chi)$, where \[N=\begin{cases}
  2^7 \cdot 3^{n+1}, \hspace{0.16in} m=0,1,2\\
  2^{m+4}\cdot 3^{n+1}, m \geq 3
\end{cases}\].
\end{proposition}

 \begin{proof} 
 Let $r=2^m \cdot 3^n$ where $m,n \in \mathbb{Z}_{\geq 0}.$\\

   \noindent First, the weight of $f_{r,k}(z)$ is:
    \[\ell=\dfrac{1}{2}\displaystyle \sum_{\delta |N} r_{\delta} = \dfrac{1}{2}\left[1+(2^k-1)-2-(2^{k-1}-2)\right]=2^{k-1}-2^{k-2}=2^{k-2}.\]

    \noindent Second, since $f_{r,k}(z)=\dfrac{\eta(48z)\eta(24rz)^{2^k-1}}{\eta(24z)^2\eta(48rz)^{2^{k-1}-2}},$ then $\delta_1=48, \delta_2=24r, \delta_3=24$ and $\delta_4=48r$ with $r_{48}=1, r_{24r}=2^k-1, r_{24}=-2,$ and $r_{48r}=2-2^{k-1}.$ \\
    
    Clearly, $f_{r,k}(z)$ satisfies equation (\ref{4}) since
`   \[\displaystyle \sum_{\delta | N} \delta r_\delta = 48\cdot 1 + 24r \cdot (2^k-1) + 24 \cdot (-2) + 48r \cdot (2-2^{k-1}) \equiv 0 \Mod{24}.\]
    
    Moreover, to satisfy equation (\ref{5}), we can let $N=48ru,$ where $u$ is the smallest positive integer satisfying 
    \[\displaystyle \sum_{\delta | N} \dfrac{N}{\delta} r_{\delta} \equiv 0 \Mod {24}.\]
    Then, 
    \begin{align*}
    \displaystyle \sum_{\delta | N} \dfrac{N}{\delta}r_\delta &=\dfrac{48ru}{48}+\dfrac{48ru}{24r}(2^k-1)-\dfrac{48ru}{24}(2)-\dfrac{48ru}{48r}(2^{k-1}-2)\\
    &=ru+2u(2^k-1)-4ru-u(2^{k-1}-2)\\
    &=u(2^{k+1}-2^{k-1}-3r)\\
    &=u(3\cdot 2^{k-1} -3r) \equiv 0 \Mod{24}
    \end{align*}

\noindent  We have the following:
\begin{itemize}
\item If $m=0,$ then $u=8$, and so $N=48 \cdot (2^0\cdot3^n) \cdot 8 =2^7\cdot3^{n+1}.$
\item If $m=1,$ then $u=4$, and so $N=48 \cdot (2^1\cdot3^n) \cdot 4 =2^7\cdot3^{n+1}.$
\item If $m=2,$ then $u=2$, and so $N=48 \cdot (2^1\cdot3^n) \cdot 4 =2^7\cdot3^{n+1}.$
\item If $m\geq 3,$ then $u=1$, and so  $N=48\cdot (2^m\cdot3^n)\cdot 1 =2^{m+4}\cdot3^{n+1}.$ 
\end{itemize}

To prove that $f_{r,k}(z) \in M_{2k-2}(\Gamma_0(N),\chi)$, it suffices to show that $f_{r,k}(z)$ is holomorphic at all cusps of $\Gamma_0(N).$\\

From Theorem \ref{ligozat}, the order of vanishing of $f_{r,k}(z)$ at the cusp $\frac{c}{d}$ where $d|N$ and gcd$(c,d)=1$, is:

\[\dfrac{N}{24} \displaystyle \sum_{\delta |N} \dfrac{\text{gcd}(d,\delta)^2r_\delta}{\text{gcd}\left(d, \frac{N}{d}\right)d\delta}\]

Hence, $f_{r,k}(z)=\dfrac{\eta(48z)\eta(24rz)^{2^k-1}}{\eta(24z)^2\eta(48rz)^{2^{k-1}-2}}$ is holomorphic at the cusp $\frac{c}{d}$ if and only if
\[    \dfrac{N}{24} \displaystyle \sum_{\delta |N} \dfrac{\text{gcd}(d,\delta)^2r_\delta}{\text{gcd}\left(d, \frac{N}{d}\right)d\delta} \geq 0  \iff \displaystyle \sum_{\delta |N} \dfrac{\text{gcd}(d,\delta)^2r_\delta}{\delta} \geq 0. \]
That is,   \[\dfrac{\text{gcd}(d,48)^2}{48}-2\dfrac{\text{gcd}(d,24)^2}{24}+(2^k-1)\dfrac{\text{gcd}(d,24r)^2}{24r}-(2^{k-1}-2)\dfrac{\text{gcd}(d,48r)^2}{48r} \geq 0.\]

\noindent Equivalently, 
     \begin{equation}
          r\text{gcd}(d,48)^2-4r\text{gcd}(d,24)^2+(2^{k+1}-2)\text{gcd}(d,24r)^2-(2^{k-1}-2)\text{gcd}(d,48r)^2 \geq 0. \label{LHS}
     \end{equation}
    
Now, if $N=2^7\cdot 3^{n+1},$ then $d=2^t\cdot 3^s, 0 \leq t \leq 7, 0 \leq s \leq n+1$. Similarly, if $N=2^{m+4}\cdot 3^{n+1},$ then $d=2^t\cdot 3^s, 0 \leq t \leq m+4, 0 \leq s \leq n+1$. \\

Let $(\star)$ be the left-hand side of inequality (\ref{LHS}). We now prove that $(\star) \geq 0$ for $k\geq m+2n+1$. We divide our proof into 6 cases.\\

\noindent \textbf{Case 1: $d=1$}\\
\noindent We have $\text{gcd}(d,48)=1$, $\text{gcd}(d,24)=1,$ $\text{gcd}(d,24r)=1$, and $\text{gcd}(d,48r)=1$.
\begin{align*}
(\star) &=(2^m\cdot3^n)-4(2^m\cdot3^n)+(2^{k+1}-2)-(2^{k-1}-2)\\
&=2^{k+1}-2^{k-1}-3\cdot(2^m\cdot 3^n)\\
&=3\cdot 2^{k-1} - 2^{m} \cdot 3^{n+1}
\end{align*}

\noindent If we let $k \geq  m+2n+1$, then
\begin{align*}
3\cdot 2^{k-1} &\geq 3\cdot 2^{m+2n}\\
&=  3\cdot 2^{m} \cdot 2^{2n}\\
&\geq  3\cdot 2^{m} \cdot 3^{n}\\
& =  2^{m} \cdot 3^{n+1},
\end{align*}

\noindent proving that $(\star) \geq 0$ for $k \geq m+2n+1$.\\

\noindent \textbf{Case 2: $d=3^s, 1 \leq s \leq n+ 1$}\\
\noindent We have $\text{gcd}(d,48)=3$, $\text{gcd}(d,24)=3,$ $\text{gcd}(d,24r)=\text{gcd}(d, 2^{m+3}\cdot3^{n+1})=3^s$, and $\text{gcd}(d,48r)=\text{gcd}(d, 2^{m+4}\cdot3^{n+1})=3^s$.
\begin{align*}
(\star) &=(2^m\cdot3^n)\cdot9 -4(2^m\cdot3^n)\cdot 9 +(2^{k+1}-2) \cdot 3^{2s}-(2^{k-1}-2)\cdot 3^{2s}\\
&=3^{2s}\cdot 3\cdot 2^{k-1} - 3 \cdot 9 \cdot (2^m\cdot3^n)\\
&=2^{k-1} \cdot 3^{2s+1} -2^m\cdot 3^{n+3}
\end{align*}

\noindent If we let $k \geq  m+2n-4s+5$, then
\begin{align*}
 2^{k-1} \cdot 3^{2s+1} & \geq 2^{m+2n-4s+4} \cdot 3^{2s+1} \\
&=  2^{m} \cdot 2^{2n-4s+4} \cdot 3^{2s+1} \\
&\geq 2^{m} \cdot 3^{n-2s+2} \cdot 3^{2s+1} \\
& =  2^{m+4} \cdot 3^{n+3},
\end{align*}

\noindent proving that $(\star) \geq 0$ for $k \geq m+2n-4s+5$. Moreover, since $m+2n-4s+5 \leq m+2n+1$, it follows that $(\star) \geq 0$ for $k \geq m+2n+1$. \\

\noindent \textbf{Case 3: $d=2^t, 0 < t \leq 3$}\\
\noindent We have $\text{gcd}(d,48)=2^t$, $\text{gcd}(d,24)=2^t,$ $\text{gcd}(d,24r)=2^t$, and $\text{gcd}(d,48r)=2^t$.
\begin{align*}
(\star) &=(2^m\cdot3^n)\cdot2^{2t} -4(2^m\cdot3^n)\cdot 2^{2t} +(2^{k+1}-2) \cdot 2^{2t}-(2^{k-1}-2)\cdot 2^{2t}\\
&=2^{2t}\cdot 3\cdot 2^{k-1} - 3 \cdot 2^{2t} \cdot (2^m\cdot3^n)\\
&=2^{k+2t-1} \cdot 3 -2^{m+2t}\cdot 3^{n+1}
\end{align*}

\noindent If we let $k \geq  m+2n+1$, then
\begin{align*}
 2^{k+2t-1} \cdot 3 & \geq 2^{m+2n+2t} \cdot 3 \\
&=  2^{m+2t} \cdot 2^{2n} \cdot 3\\
&\geq 2^{m+2t} \cdot 3^{n} \cdot 3 \\
& =  2^{m+2t} \cdot 3^{n+1},
\end{align*}

\noindent proving that $(\star) \geq 0$ for $k \geq m+2n+1$.\\

\noindent \textbf{Case 4: $d=2^t, 3<  t \leq m+4$}\\
\noindent We have $\text{gcd}(d,48)=2^4$, $\text{gcd}(d,24)=2^3,$ $\text{gcd}(d,24r)=\text{gcd}(d, 2^{m+3}\cdot3^{n+1})=2^t$, and $\text{gcd}(d,48r)=\text{gcd}(d, 2^{m+4}\cdot3^{n+1})=2^t$.
\begin{align*}
(\star)&=(2^m\cdot3^n)\cdot2^{8} -4(2^m\cdot3^n)\cdot 2^{6} +(2^{k+1}-2) \cdot 2^{2t}-(2^{k-1}-2)\cdot 2^{2t}\\
&=2^{2t}\cdot 3\cdot 2^{k-1}\\
&=2^{k+2t-1} \cdot 3 \geq 0 \text{ for } k \geq 1
\end{align*}

\noindent \textbf{Case 5: $d=2^t \cdot 3^s, 0 < t \leq 3, 1 \leq s \leq n+1$}\\
\noindent We have $\text{gcd}(d,48)=2^t\cdot 3$, $\text{gcd}(d,24)=2^t \cdot 3,$ $\text{gcd}(d,24r)=\text{gcd}(d, 2^{m+3}\cdot3^{n+1})=2^t \cdot 3^s$, and $\text{gcd}(d,48r)=\text{gcd}(d, 2^{m+4}\cdot3^{n+1})=2^t \cdot 3^s$.
\begin{align*}
(\star)=&(2^m\cdot3^n)\cdot(2^{2t}\cdot 3^2) -4(2^m\cdot3^n)\cdot(2^{2t}\cdot 3^2)+(2^{k+1}-2)\cdot(2^{2t}\cdot 3^{2s})\\
&-(2^{k-1}-2)\cdot(2^{2t}\cdot 3^{2s})\\
=&(2^{2t}\cdot 3^{2s}) \cdot 3\cdot 2^{k-1} - 3 \cdot(2^{2t}\cdot 3^2) \cdot (2^m\cdot3^n)\\
=&2^{k+2t-1} \cdot 3^{2s+1} -2^{m+2t}\cdot 3^{n+3}
\end{align*}

\noindent If we let $k \geq  m+2n-4s+5$, then
\begin{align*}
2^{k+2t-1} \cdot 3^{2s+1} & \geq 2^{m+2n+2t-4s+4} \cdot 3^{2s+1} \\
&=  2^{m+2t} \cdot 2^{2n-4s+4} \cdot 3^{2s+1}\\
&\geq 2^{m+2t} \cdot 3^{n-2s+2} \cdot 3^{2s+1} \\
& =  2^{m+2t} \cdot 3^{n+3},
\end{align*}

\noindent proving that $(\star) \geq 0$ for $k \geq m+2n-4s+5$. Moreover, since $m+2n-4s+5 \leq m+2n+1$, it follows that $(\star) \geq 0$ for $k \geq m+2n+1$.\\

\noindent \textbf{Case 6: $d=2^t \cdot 3^s, 3 < t \leq m+4, 1 \leq s \leq n+1$}\\
\noindent We have $\text{gcd}(d,48)=2^4\cdot 3$, $\text{gcd}(d,24)=2^3 \cdot 3,$ $\text{gcd}(d,24r)=\text{gcd}(d, 2^{m+3}\cdot3^{n+1})=2^t \cdot 3^s$, and $\text{gcd}(d,48r)=\text{gcd}(d, 2^{m+4}\cdot3^{n+2})=2^t \cdot 3^s$.
\begin{align*}
(\star)=&(2^m\cdot3^n)\cdot(2^8\cdot 3^2) -4(2^m\cdot3^n)\cdot(2^6\cdot 3^2)+(2^{k+1}-2)\cdot(2^{2t}\cdot 3^{2s})\\
&-(2^{k-1}-2)\cdot(2^{2t}\cdot 3^{2s})\\
=&(2^{2t}\cdot 3^{2s}) \cdot 3\cdot 2^{k-1}\\
=&2^{k+2t-1}\cdot 3^{2s+1} \geq 0 \text{ for } k \geq 1.
\end{align*}

In all possible cases, we have that $(\star) \geq 0$ for $k \geq m+2n+1$ where $k$ is a positive integer greater than 3. Hence, by Theorem \ref{ligozat}, $f_{r,k}(z)$ is a modular form of weight $2^{k-2}.$

\end{proof}
Lastly, we will use Serre's Theorem \cite{serre} regarding the coefficients of the Fourier expansion of a holomorphic modular form to prove our final result.

\begin{theorem} [Serre]
Let $k,m$ be positive integers. If $f(z) \in M_k(\Gamma_0(N), \chi)$ has Fourier expansion $f(z) = \sum_{n=0}^{\infty} c(n)q^n \in \mathbb{Z}[[q]],$ then there is a constant $\alpha >0$ such that 
\[\#\{n\leq X: c(n) \not\equiv 0\mod m\} = \mathcal{O}\left(\dfrac{X}{\log^\alpha X}\right).\]
\end{theorem}

\noindent \large \textbf{Proof of Theorem \ref{main}}
\normalsize
\begin{proof}
Let $r = 2^m \cdot 3^n$ where $m,n \in \mathbb{Z}_{\geq 0}$ and  $k \geq m+2n+1, \, k \geq 3$ be a positive integer.  Since $f_{r,k}(z) \in M_{2^{k-2}}(\Gamma_0(N),\chi)$ and the Fourier coefficients of $f_{r,k}(z)$ are integers, then by Serre's Theorem, we can find a constant $\alpha >0$ such that 
\[\#\{n\leq X: \sigma_r\overline{mex}(n) \not\equiv 0\mod 2^k\} = \mathcal{O}\left(\dfrac{X}{\log^\alpha X}\right),\]
for $k \geq m+2n+1.$\\

Then 
\[\displaystyle \lim_{X \to +\infty} \dfrac{\#\{n\leq X: \sigma_r\overline{mex}(n) \equiv 0\mod 2^k\}}{X}=1.\]

Equivalently, for  almost every nonnegative integer $n$ lying in an arithmetic progression, $\sigma_r\overline{mex}(n)$ is a multiple of $2^k$ where $r = 2^m \cdot 3^n$ where $m,n \in \mathbb{Z}_{\geq 0}$ and  $k \geq m+2n+1, k \geq 3$. Consequently, $\sigma_r\overline{mex}(n)$ is a multiple of $2^k$, where $k \geq 1$.
\end{proof}

\vspace*{12pt}

\end{document}